\documentclass[11pt, a4paper, reqno]{amsart} 
\usepackage[left=2.5cm, right=2.5cm, bottom=3cm]{geometry}
\usepackage{graphicx} 
\usepackage{amsmath}
\usepackage{amssymb}
\usepackage{amsthm}
\usepackage{graphicx}
\usepackage{subcaption}
\usepackage{float}
\usepackage{bbm}
\usepackage{mathabx}
\usepackage{caption}
\usepackage{hyperref}
\usepackage{mathtools}
\usepackage[dvipsnames]{xcolor}

\newtheorem{theorem}{Theorem}
\newtheorem{lemma}{Lemma}

\newtheorem{corollary}{Corollary}

\theoremstyle{remark}

\title{A note on the norm of the embedding of $\text{PW}^2$ into $\text{PW}^q$}
\author{Iñaki Garrido, Filippo Persico, Denis Zelent}
\address{Iñaki Garrido Pérez \newline Departament de Matemàtiques i Informàtica, Universitat de Barcelona,
Barcelona, Spain} 
\email{inaki.garrido@ub.edu}
\address{Filippo Persico \newline Dipartimento di Matematica, Università
degli Studi di Milano, Milano, Italy \newline Departament de Matemàtiques i Informàtica, Universitat de Barcelona,
Barcelona, Spain} 
\email{filippo.persico@unimi.it}
\address{Denis Zelent \newline Department of Mathematical Sciences, Norwegian University of Science and Technology (NTNU), 7491 Trondheim, Norway} 
\email{denis.zelent@ntnu.no}

\thanks{Denis Zelent was supported by Grant 334466 of the Research Council of Norway.}

\begin{document}

\begin{abstract}
   We improve the currently best known bound for the norm of the embedding of $PW^2$ into $PW^q$ for all $q>2$. We show that our bound is asymptotically optimal as $q\to \infty$. Our approach combines the recently established theorem of optimal Paley--Wiener concentration with a Hardy--Littlewood--Pólya majorization argument.
\end{abstract}
\maketitle
\section{Introduction}
In this paper, we are interested in the value of $M_q$, where $M_q$ is defined for $2\leq q < \infty$ as 
$$ M_q := \sup_{\substack{f\in PW^2\\ \|f\|_2=1}}
\int_{\mathbb{R}}|f(x)|^q dx. $$
Here
$$ PW^q := \left\{f\in L^q(\mathbb{R}): \operatorname{supp}\widehat f\subset[-1/2,1/2] \right\} $$
is the Paley--Wiener space, and the Fourier transform of $f$ is defined as 
$$ \widehat f(\xi) := \int_{\mathbb{R}} f(x)e^{-2\pi i x\xi} dx.$$
Equivalently,
$$ M_q^{1/q} = \sup_{0\neq f\in PW^2}
\frac{\|f\|_q}{\|f\|_2}, $$
which represents the norm of the embedding of $PW^2$ into $PW^q$.

In general, for $p\leq q$, the norm of the embedding of $PW^p$ into $PW^q$, represented by 
$$ \sup_{0\neq f\in PW^p} \frac{\|f\|_q}{\|f\|_p},$$
is a particular case of the Bernstein-Nikolskii inequalities, which have been studied extensively in the literature (see, e.g., \cite{Ganzburg1} for a survey). Related extremal and optimization problems for Paley-Wiener spaces have also attracted considerable attention in recent years; see, for instance, \cite{Brevig, CarneiroEMB, CarneiroFOPT, instanes2024}.
Nevertheless, even in the case of the embedding $PW^2 \xhookrightarrow{} PW^q$, the sharp constant $M_q$ remains unknown for general $q$. We next summarize the current state of knowledge concerning~$M_q$.

Trivially, $M_2 = 1$. 
If we define 
$$ M_{\infty} = \sup_{\substack{f\in PW^2\\ \|f\|_2=1}}\|f\|_{\infty}, $$ then $M_{\infty}=1$ follows easily by the fact that for $f\in PW^2$ (using Cauchy--Schwarz)
$$|f(x)| = \left|\int_{-1/2}^{1/2}\hat{f}(\xi)e^{2\pi i x \xi}d\xi\right| \leq \|\hat{f}\|_2 = \|f\|_2.$$
Taking $f(x) = \frac{\sin \pi(x-y)}{\pi(x-y)}$ shows that the constant is optimal.\\
\textsl{Remark.} Note that this case is different from what we mean in this paper by the constant $M_q$ as $q\to \infty$, i.e., $\lim_{q\to \infty} M_q \neq M_{\infty}$.

From \cite[Sec. 4.9.53]{Timan1963} it follows that 
$$M_q \leq 1.$$
See also \cite[Eq. (1.14)]{Ganzburg1}, where we note that the difference comes from using different normalization in the Fourier transform convention. 

The only case other than $q=2,\infty$, when the sharp constant $M_q$ is known with high precision is $q=4$. It follows from \cite{Garsia} that
    $$M_4 = C_0,$$
where $C_0 = .686981293033114600949413\ldots$. See also \cite[Eq. (1.16)]{Ganzburg1}.

The best bound for a general $q>2$, to the best of our knowledge, follows from the Babenko--Beckner inequality \cite{Beckner}, which says that for $f\in L^p(\mathbb{R})$, $p\in [1,2]$, $1/p+1/p'=1$, we have
$$\left(\int_{\mathbb{R}}|\hat{f}(\xi)|^{p'}d\xi\right)^{1/p'} \leq \sqrt{\frac{p^{1/p}}{p'^{1/p'}}} \left(\int_{\mathbb{R}}|f(x)|^pdx\right)^{1/p}.$$
Applying it to $f=\check{g}$  with $p'=q$ and using Hölder's inequality on the right-hand side with the fact that $f\in PW^2$ gives 
$$\left(\int_{\mathbb{R}}|g(x)|^q dx\right)^{1/q} \leq \sqrt{\frac{q'^{1/q'}}{q^{1/q}}} \|g\|_2,$$
and therefore 
$$ M_q^{1/q} \leq \sqrt{\frac{q'^{1/q'}}{q^{1/q}}},$$
or equivalently 
\begin{align}\label{BabenkoBeckner}
    M_q \leq B_q := \sqrt{\frac{q^{q-2}}{(q-1)^{q-1}}}.
\end{align}
Note that as $q\to\infty$, the asymptotic behaviour of the Babenko--Beckner upper bound is
$$B_q \sim \sqrt{\frac{e}{q}}.$$

It can be easily seen that the above bound is not asymptotically optimal as $q\to \infty$. Indeed, it is clear that we have
$$M_q \geq \int_{\mathbb{R}} \left|\frac{\sin(\pi x)}{\pi x}\right|^q dx.$$
Using that $\frac{\sin(\pi x)}{\pi x} \geq 1-\frac{\pi^2 x^2}{6}$ on $(-\frac{\sqrt{6}}{\pi}, \frac{\sqrt{6}}{\pi})$ this gives
\begin{align}\label{lowbound}
    M_q \geq \sqrt{\frac{6}{\pi}} \frac{\Gamma(q+1)}{\Gamma(q+3/2)} \sim \sqrt{\frac{6}{\pi q}}.
\end{align}

In this paper, we want to improve the general bound coming from the Babenko--Beckner inequality, focusing mainly on the case $q\to \infty$. Our main theorem is as follows.
\begin{samepage}
\begin{theorem}\label{mainth}
For any $q>2$
    \begin{align*}
    M_q = \sup_{\substack{f\in PW^2\\ \|f\|_2=1}}\|f\|_q^q \leq \Lambda_q := \int_0^{\infty} (\lambda'(x))^{q/2}dx,
    \end{align*}
where $\lambda(s)$ denotes the best $L^2$ concentration in $PW^2$ on $[-s/2,s/2]$, i.e.,
$$\lambda(s):=\sup_{\substack{f\in PW^2\\ \lVert f\rVert_2=1}}
 \int_{-s/2}^{s/2}|f(x)|^2 dx.$$
\end{theorem}
\end{samepage}

Figure \ref{fig:upper-bound-comparison} shows how the upper bound $\Lambda_q$ improves the Babenko--Beckner upper bound $B_q$ from equation \eqref{BabenkoBeckner} for all $q > 2$.
\begin{figure}[H]
    \centering
    \includegraphics[width=0.8\textwidth]{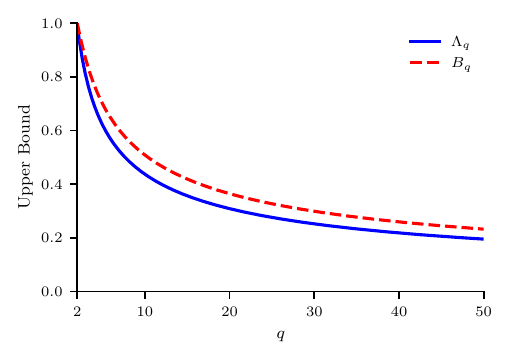}
    \caption{Comparison of the Babenko--Beckner upper bound $B_q$ and the improved upper bound $\Lambda_q$.}
    \label{fig:upper-bound-comparison}
\end{figure}

The following lemmas allow us to study the asymptotic behavior of our bound as $q\to \infty$.
\begin{lemma}\label{smalls}
As $s\to 0^{+}$, we have
\begin{align*}
    \lambda(s) &= s - \frac{\pi^2}{36}s^3 + O(s^5),\\
    \lambda'(s) &= 1 - \frac{\pi^2}{12}s^2 + O(s^4).
\end{align*}
\end{lemma}
\begin{lemma}\label{lambda(s)/s}
For all $s> 0$ we have that $\lambda(s)/s$ is non-increasing and 
    \begin{align*}
   \lambda'(s) \leq \frac{\lambda(s)}{s}.
\end{align*}
\end{lemma}
Using the above results, we obtain the following upper bound as $q\to\infty$.
\begin{corollary}\label{upperbound}
We have, as $q\to\infty$,
\begin{align*}
    \Lambda_q \leq \sqrt{\frac{6}{\pi q}} + O(q^{-3/2}).
\end{align*}
\end{corollary}

Corollary \ref{upperbound}, together with eq. (\ref{lowbound}), shows that Theorem \ref{mainth} is asymptotically sharp as $q\to \infty$, and so, as $q\to\infty$,
$$M_q \sim \sqrt{\frac{6}{\pi q}}.$$

This short note is organized as follows. Section $2$ presents the proof of Theorem \ref{mainth}. Section $3$ is devoted to the results about the function $\lambda$ and includes proofs of Lemmas \ref{smalls} and \ref{lambda(s)/s}, and Corollary \ref{upperbound}.

\section{Proof of Theorem \ref{mainth}}
Our proof idea for Theorem \ref{mainth} relies mainly on using the optimal concentration in the Paley--Wiener space. For sets of size $0 < |S| < 0.8$, this is the classic Donoho--Stark result \cite{DonohoStark}. It was conjectured in 1989 (see Conjecture 1 in \cite{DonohoStarkConjecture}) that the result is true for all $0 < |S| < \infty$, which was only proven very recently in \cite{abreu2026optimalconcentrationpaleywienerspace}:

\begin{theorem}[Optimal Paley–Wiener concentration]
\label{OptimalPWConcentration}
Let $S \subset \mathbb{R}$ be a measurable set with $0 < |S| < \infty$ and $g$ be any function of $PW^2$ with $\|g\|_{2}=1$. Then
\begin{align*}
     \int_{0}^{|S|}u^{*}(x) dx \leq \sup_{\substack{f\in PW^2\\ \|f\|_2=1}} \int_{-|S|/2}^{|S|/2}|f(x)|^2 dx = \lambda(|S|).
\end{align*}
Here $u^{*}$ denotes the decreasing rearrangement of $|g|^2$.
\end{theorem}

We then combine it with the following lemma, which is an extension of a theorem of Hardy, Littlewood and Pólya that can be found in \cite[Theorem 2.1]{Chong_1974}.
\begin{lemma}[Hardy--Littlewood--Pólya majorization principle]\label{HLP}
    Assume that $f,g: [0, \infty) \xrightarrow{} \mathbb{R}^{+}$ are non-increasing functions with 
    \begin{equation}
    \label{eq:majorization}
    \int_0^t f(x) dx \leq \int_0^t g(x)dx
    \end{equation}
    for all $0<t< \infty$.
    Then 
    $$\int_0^{\infty} \Phi(f(x))dx \leq \int_0^{\infty} \Phi(g(x))dx$$
    holds for every increasing convex function $\Phi: \mathbb{R}^{+} \xrightarrow{} \mathbb{R}^{+}$ such that $\Phi(0) = 0$.
\end{lemma}

\begin{proof}[Proof of Theorem \ref{mainth}]
We use first Theorem \ref{OptimalPWConcentration} and introduce the decreasing rearrangement of $\lambda'$ to obtain that 
\begin{align*}
     \int_{0}^{s}u^{*}(x) dx \leq \lambda(s) =  \int_{0}^s \lambda'(x)dx \leq \int_{0}^s (\lambda')^{*}(x)dx.
\end{align*}
holds for all $s > 0$. 
Note that we introduce the decreasing rearrangement of $\lambda'$ here to be able to use Lemma \ref{HLP} without proving that $\lambda'$ itself is non-increasing. 
We can thus apply Lemma \ref{HLP} with $\Phi(x) = x^{q/2}$, $q\geq 2$, to obtain
\begin{align*}
     \int_{0}^{\infty}(u^{*}(x))^{q/2} dx \leq \int_{0}^{\infty} ((\lambda')^{*}(x))^{q/2}dx.
\end{align*}
Using now the properties of the decreasing rearrangement we conclude that 
\begin{align*}
    \int_{\mathbb{R}}|f(x)|^q dx &= \int_0^{\infty} (u^{*}(x))^{q/2}dx \leq \int_{0}^{\infty} ((\lambda')^{*}(x))^{q/2}dx = \int_{0}^{\infty} (\lambda'(x))^{q/2}dx.
    \end{align*}
\end{proof}

\textbf{Remark.}
We want to highlight that although this method gives an asymptotically optimal bound as $q\to\infty$, we don't expect it to be optimal for any fixed finite $q$. For example, from \cite{Garsia} we know that $M_4 = 0.68698...$, but the bound we get with our method is $M_4\leq \int_0^{\infty} (\lambda'(x))^{2}dx \approx 0.6969$.

\section{Results about the function \texorpdfstring{$\lambda$}{lambda}}
In this section, we prove the results previously mentioned about the function $\lambda$.
It is well known that $\lambda$ is a real-analytic function \footnote{This fact can be justified as follows. Let \(S_{00}(c,x)\) denote the first angular prolate spheroidal function as in \cite{SlepianPollakI}. The results of \cite[\S\S 3.21--3.23, pp.~230--238]{MeixnerSchaefke1954} give its local joint analyticity in \((c,x)\) near \(\mathbb{R}\times\mathbb{C}\). Equation~(25) of \cite{SlepianPollakI}, evaluated at zero, then transfers this regularity to \(R_{00}^{(1)}(c,1)\) by integration over \([-1,1]\), and Equation~(27) of \cite{SlepianPollakI} yields a real-analytic extension of \(\lambda\) through \(c=0\).} such that
$$\lambda(s) = \int_{-s/2}^{s/2}|\psi_s(x)|^2 dx,$$
where $\psi_s$ is an appropriately scaled prolate spheroidal wave function.
Another useful fact is that if $R_{0,0}^{(1)}$ is a suitable radial prolate spheroidal function, then 
$$\lambda(s) = s \left(R_{0,0}^{(1)}\left(\frac{\pi s}{2},1\right)\right)^2.$$
See e.g. \cite{LandauPollakII,LandauPollakIII,SlepianPollakI} for the seminal work on the eigenfunctions and eigenvalues of the time-frequency concentration operator. 
\begin{proof}[Proof of Lemma \ref{smalls}]
    We will use that 
    $$\lambda(s) = s \left(R_{0,0}^{(1)}\left(\frac{\pi s}{2},1\right)\right)^2,$$
    where $R_{0,0}^{(1)}$ is a radial prolate spheroidal function which can be expanded using \cite[eq. (15.3.5)]{Zhang} as
    \begin{align*}
        R_{0,0}^{(1)}\left(\frac{\pi s}{2},1\right) = \left(\sum_{k=0}^{\infty}d_{2k}\left(\frac{\pi s}{2}\right)\right)^{-1} \sum_{k=0}^{\infty} (-1)^k d_{2k}\left(\frac{\pi s}{2}\right)j_{2k}\left(\frac{\pi s}{2}\right).
    \end{align*}
    Here $j_{2k}$ denotes the spherical Bessel function and $d_{2k}$ can be computed using \cite[eq. (15.4.7)]{Zhang}.
    As it is well known (see e.g. \cite{SphericalBessel}), 
    $$j_{n}(z) = z^n \sum_{k\geq 0}\frac{(-1)^k}{k! (2k+2n+1)!!}\left(\frac{z^2}{2}\right)^k,$$
    which means that 
    $$j_{2k}\left(\frac{\pi s}{2}\right) = \frac{1}{(4k+1)!!}\left(\frac{\pi s}{2}\right)^{2k} + O\left(s^{2k+2}\right).$$
    One can also compute that 
    \begin{align*}
        d_0(\pi s/2) &= 1 - \frac{1}{18}(\pi s/2)^2 + O(s^4), \\
        d_2(\pi s/2) &= -\frac{1}{9}(\pi s/2)^2 + O(s^4), \\
        d_{2k}(\pi s/2) &= O(s^{2k}),
    \end{align*}
    which altogether give
    \begin{align*}
        \lambda(s) &= s\left(\frac{(1 - \frac{\pi^2 s^2}{72} + O(s^4))(1-\frac{\pi^2 s^2}{24}+O(s^4)) + O(s^4)}{1-\frac{\pi^2 s^2}{24} + O(s^4)}\right)^2 \\
        &= s\left(1 - \frac{\pi^2 s^2}{72} + O(s^4)\right)^2 \\
        &= s- \frac{\pi^2}{36}s^3 + O(s^5).
    \end{align*}

    Since \(\lambda\) is real-analytic in a neighborhood of \(0\), the expansion of \(\lambda'\) follows immediately by termwise differentiation of the expansion above. 
\end{proof}

\begin{proof}[Proof of Lemma \ref{lambda(s)/s}]
Fix $0<a<b$ and let $\psi_b$ be such that $\lambda(b) = \int_{-b/2}^{b/2}|\psi_b(x)|^2dx$. Choose a measurable set \(E\subset (-\tfrac{b}{2}, \tfrac{b}{2})\), with \(|E|=a\), such that \(|\psi_b(x)|^2\geq |\psi_b(y)|^2\) for almost every \(x\in E\) and \(y\in (-\tfrac{b}{2}, \tfrac{b}{2})\setminus E\). Then clearly 
\begin{align*}
    \frac{1}{a}\int_E |\psi_b(x)|^2dx \geq \frac{1}{b}\int_{-b/2}^{b/2} |\psi_b(x)|^2dx.
\end{align*}
Therefore we get that 
\begin{align*}
    \lambda(a) \geq \int_E |\psi_b(x)|^2dx \geq \frac{a}{b}\int_{-b/2}^{b/2} |\psi_b(x)|^2dx = \frac{a}{b}\lambda(b),
\end{align*}
from which it follows that $\frac{1}{s}\lambda(s)$ is non-increasing for $s> 0$.
This leads us to
\begin{align*}
    \left(\frac{\lambda(s)}{s}\right)' = \frac{\lambda'(s)s-\lambda(s)}{s^2} \leq 0,
\end{align*}
or equivalently 
\begin{align*}
   \lambda'(s) \leq \frac{\lambda(s)}{s}
\end{align*}
for all $s\geq 0$.
\end{proof}

\begin{proof}[Proof of Corollary \ref{upperbound}]
    By Lemma~\ref{smalls} there exist $C\geq 1$ and $\delta\in(0,1]$ such that, for all $0<x\leq\delta$,
\begin{equation}\label{eq:deltaC}
\lambda'(x)\leq 1-\frac{\pi^{2}}{12}x^{2}+Cx^{4},
\qquad
\frac{\lambda(x)}{x}\leq 1-\frac{\pi^{2}}{72}x^{2},
\qquad
C\delta^{2}\leq\frac{\pi^{2}}{24}.
\end{equation}
Since $1+t\leq e^{t}$, the first bound in \eqref{eq:deltaC} gives
$(\lambda'(x))^{q/2}\leq \exp\bigl(-\tfrac{\pi^{2}q}{24}x^{2}+\tfrac{Cq}{2}x^{4}\bigr)$
on $(0,\delta]$. Using $e^{v}\leq 1+ve^{v}$ for $v\geq 0$ together with
$\tfrac{Cq}{2}x^{4}\leq \tfrac{\pi^{2}q}{48}x^{2}$, which follows from the third bound
in \eqref{eq:deltaC}, we obtain
\begin{align}
\int_{0}^{\delta}\bigl(\lambda'(x)\bigr)^{q/2}dx
&\leq \int_{0}^{\infty} e^{-\frac{\pi^{2}q}{24}x^{2}}dx
	+\frac{Cq}{2}\int_{0}^{\infty} x^{4}\,e^{-\frac{\pi^{2}q}{48}x^{2}}dx \notag\\
&= \sqrt{\frac{6}{\pi q}}
	+\frac{Cq}{2}\cdot\frac{3\sqrt{\pi}}{8}\Bigl(\frac{\pi^{2}q}{48}\Bigr)^{-5/2}
= \sqrt{\frac{6}{\pi q}}+O\bigl(q^{-3/2}\bigr).
\label{eq:mainrange}
\end{align}
To bound the integral on $(\delta, \infty)$ we use Lemma \ref{lambda(s)/s} and the fact that $\int_{0}^{\infty} \lambda'(x)dx = 1$ to get that
\begin{align}\label{eq:awayrange}
    \int_{\delta}^{\infty}\left(\lambda'(x)\right)^{q/2}dx \leq \left(\frac{\lambda(\delta)}{\delta}\right)^{q/2-1} \int_{\delta}^{\infty} \lambda'(x)dx \leq \left(1-\frac{\pi^2}{72}\delta^2\right)^{q/2-1},
\end{align}
which decays exponentially in $q$.
Adding \eqref{eq:mainrange} and
\eqref{eq:awayrange} proves the claim.
\end{proof}

\thanks{\textbf{Acknowledgments} The main idea behind this note was developed during a research visit of FP and DZ to the Universitat de Barcelona. They thank the Departament de Matemàtiques i Informàtica for its warm hospitality. All three authors are grateful to Joaquim Ortega-Cerdà, who supervised the visit, for his guidance and for many valuable discussions. }
\bibliographystyle{abbrv}
\bibliography{papers}

\end{document}